\documentclass[twoside,a4paper]{amsart}

\usepackage{pdfsync}

\usepackage{xcolor}
\definecolor{webgreen}{rgb}{0,.5,0}
\definecolor{webbrown}{rgb}{.6,0,0}
\definecolor{RoyalBlue}{cmyk}{1, 0.50, 0, 0}
\usepackage[colorlinks=true, breaklinks=true, urlcolor=webbrown, linkcolor=RoyalBlue, citecolor=webgreen,backref=page]{hyperref}

\usepackage{epsfig, graphicx, subfigure}
\usepackage{amsmath, amssymb}

\normalfont
\usepackage{newtxtext,newtxmath}

\usepackage{mathabx}

\usepackage{mathtools}
\mathtoolsset{showonlyrefs}

\newcommand{\R}     {\mathbb{R}}
\newcommand{\C}     {\mathbb{C}}
\newcommand{\N}     {\mathbb{N}}

\newcommand{\supp}{\mathrm{supp}}

\newcommand{\sgn}{\mathrm{sgn}}

\renewcommand{\det}{\mathrm{det}}

\newcommand{\ic}{\mathrm{i}}

\def\ge{\geqslant}
\def\le{\leqslant}

\newtheorem{theorem}{Theorem}
\newtheorem{proposition}[theorem]{Proposition}
\newtheorem{corollary}[theorem]{Corollary}

\theoremstyle{remark}

\begin{document}

\title[On AAA approximants]{On AAA approximants on an interval to Nevanlinna-Pick functions and P\'olya Frequency Series}

\begin{abstract}
We study the connection between AAA (triple A) algorithm and  multi-point Pad\'e approximation when the approximated functions belong to either the Nevanlinna-Pick class or are P\'olya frequency series. 
\end{abstract}

\author{Sergey A. Denisov}
\address{Department of Mathematics, University of Wisconsin-Madison, 480~Lincoln Dr., Madison, WI 53706, USA}
\email{\href{mailto:denissov@math.wisc.edu}{denissov@math.wisc.edu}}

\author{Maxim L. Yattselev}

\address{Department of Mathematical Sciences, Indiana University Indianapolis, 402~North Blackford Street, Indianapolis, IN 46202}

\email{\href{mailto:maxyatts@iu.edu}{maxyatts@iu.edu}}

\thanks{The research of the first author was supported by the  NSF grant DMS-2450716, the Simons Fellowship in Mathematics, the Simons Travel Support for Mathematicians Award, and the Van Vleck Professorship Research Award.  The second author completed this work during his tenure as Royal Society Wolfson Visiting Fellow (RSWVF\textbackslash R3 \textbackslash 253003) at the School of Mathematics in Bristol University where part of this work was done.} 



\maketitle

\section{Introduction and Main Results}

Rational approximation of analytic functions is a well established subject \cite{BakerGraves-Morris,Gaier,PetrushevPopov,StahlTotik,Walsh}. A lot of the theoretical literature is concerned with multi-point Pad\'e approximants (rational functions whose coefficients are chosen to satisfy the maximal possible number of interpolation conditions) and/or best rational approximants (rational functions whose coefficients are chosen to minimize a given norm on a given set), see \cite{uBarYa,BarStYa12,GonRakh87} and references therein.  These theoretical methods yield corresponding numerical algorithms that vary in degree of efficiency and difficulty of implementation, see \cite{Meinardus,Trefethen} or \url{https://project.inria.fr/rarl2/}. In recent years, Trefethen and collaborators \cite{MR3805855,MR5088237} have popularized the use of the AAA algorithm, which appears to be computationally fast and numerically accurate. The AAA algorithm can be viewed as a mixture of interpolation and norm-minimization approaches. Indeed, let \( f(x) \) be a smooth real-valued function on an interval $[a,b]$ (in what follows, we only consider the case where the interpolation  takes place on an interval). Denote by \( K_f(x,y) \) the corresponding Loewner kernel, defined as
\begin{align}
\label{Loewner kernel}
K_f(x,y) := \frac{f(x)-f(y)}{x-y},
\end{align}
which is a continuous function on $[a,b]\times[a,b]$ due to the  smoothness of \( f(x) \). Let  \( {\nu_n} \) be a finite Borel measure with at least \( n \) points in its support, which is a subset of $[a,b]$, and \( T_n\subset [a,b] \) be a collection of \( n \) distinct points:
\begin{align}
\label{nu_n and T_n}
\supp\,\nu_n\subseteq [a,b], \;\; \#\,\supp\,\nu_n\geq n, \quad \text{and} \quad T_n=\{a\leq t_1<t_2<\cdots<t_n\leq b\}.
\end{align}
 The \emph{AAA approximant of \( f(x) \) corresponding to the pair \( (\nu_n,T_n) \)} is a rational function \( R_n(x) \) of type \( (n-1,n-1) \),  defined by
\begin{align}
\label{AAA}
R_n(x) = \left( \sum_{k=1}^n \frac{b_k f(t_k)}{x-t_k} \right) / \left( \sum_{k=1}^n \frac{b_k}{x-t_k} \right),
\end{align}
where the real coefficients  \( b_k \), the barycentric weights, satisfy \( b_1^2+\cdots+b_n^2=1 \) (we normalize the non-zero coefficient $b_k$ with the smallest index $k$ to be positive) and are chosen so that
\begin{equation}
\label{L2-min}
\int\left( \sum_{k=1}^n b_k K_f(x,t_k) \right)^2 d{\nu_n}(x) = \min_{c_1^2+\cdots+c_n^2=1}\int\left( \sum_{k=1}^n c_k K_f(x,t_k) \right)^2 d{\nu_n}(x).
\end{equation}
Clearly, such $b_1,\ldots,b_n$ exist. Notice that \( R_n(t_k)=f(t_k) \) if \( b_k\neq 0 \), \( 1\leq k\leq n\). Hence, \( R_n(x) \) is a rational interpolant of \( f(x) \) written in the barycentric form with its barycentric coordinates chosen so that the linearized error of approximation
\begin{align}
\label{lin error}
f(x)\sum_{k=1}^n \frac{c_k}{x-t_k} - \sum_{k=1}^n \frac{c_k f(t_k)}{x-t_k} = \sum_{k=1}^n c_k K_f(x,t_k)
\end{align}
is minimized in the \( L^2(\nu_n)\)-norm. In the AAA algorithm, see \cite{AA,MR3805855,MR5088237}, one starts with a finite set $X\subset [a,b]$ on which $f(x)$ is given. Then, one recursively chooses sets \( T_1\subset T_2\subset\ldots \subset T_n\subset X \) in a greedy fashion. That is, for known $T_m=\{t_i\}_{i=1}^m$, the next interpolation point $t_{m+1}$ is chosen to maximize the approximation error
\[
t_{m+1} = \mathrm{arg \, max}_{x\in X} |f(x)-R_{m}(x)|.
\]
Here, the approximant $R_m(x)$ is defined by \eqref{AAA} and \eqref{L2-min}, where  \( {\nu_m}=\sum_{x\in X\backslash T_m} \delta_x \).

In another connection, if \( f\in\mathcal H[a,b] \) -- that is, \( f(x) \) extends to an analytic functions in some  domain $\mathcal{D}\supset[a,b]$ -- then, as mentioned before, there is another well-studied type of rational interpolation: multi-point Pad\'e approximation. Recall that \emph{a multi-point Pad\'e approximant of \( f(z) \) in a bounded domain $\mathcal{D}$} is a rational function \( q(z)/p(z) \) of type \((n-1,n-1)\) for which there exists a polynomial \( v(z) \) of degree at least \( 2n-1 \) with all zeros in \( \mathcal{D} \) such that the linearized error \( (pf-q)(z)/v(z)\) is also analytic in \( \mathcal{D} \) (if \( \mathcal D \) is unbounded, this definition should be modified to allow the interpolation conditions be placed at infinity as well). 

Let us emphasize that the coefficients of the polynomials forming a AAA approximant are chosen to try interpolating at given \( n \) points while minimizing a certain \( L^2\)-norm, whereas the coefficients of the polynomials forming a multi-point Pad\'e approximant are chosen to try interpolating at \( 2n-1 \) points in Hermite sense. (We say ``try'' because interpolation in AAA approximation can fail due to the possible vanishing of some  barycentric weights $b_k$, while interpolation in Pad\'e approximation can fail due to possible common zeros of \( p(z) \), \( q(z) \) and \( v(z) \).) Observe also that in both cases it is not a priori clear where the poles of the approximants are located.

Currently, the rigorous theory of multi-point Pad\'e approximation is developed better owing in part to its connection to orthogonal polynomials. On the other hand, the AAA algorithm is an effective and established method of numerical rational approximation \cite{MR5088237}.  Our current study is motivated by the following question: 
\smallskip

\emph{When is a AAA approximant on an interval  also a multi-point Pad\'e approximant?} 
\smallskip

\noindent
Below we show that this implication does take place for  two classes of functions: \emph{Nevanlinna-Pick functions} and \emph{P\'olya frequency series}. This modest list is most certainly not exhaustive. In Section~\ref{sec:loewner}, we explain how we arrived at these two classes and how to look for other ones to which our approach could apply.

Let us now recall some definitions. A function \( F(z) \) is called a \emph{Nevanlinna-Pick function} if it maps $\C^+$ (the upper half-plane) into itself and is analytic on $\C^+$. Such functions are characterized by the integral representation
\begin{align}
\label{N-P}
F(z) = A+ Bz + \int\frac{1+zs}{s-z}d\mu(s),
\end{align}
where \(\mu\) is a finite Borel measure on the real line, \( A \) is real, and \( B\geq 0 \), see \cite{Simon}. Clearly, \( F(z) \) is a rational function if and only if the support of \( \mu \) has finite cardinality. This class, in particular, includes such functions as \( \log z\) and \( z^\alpha \) with \( \alpha\in (0,1) \). {\it Markov functions} are defined by the formula
\[
\widehat\sigma(z) := \int\frac{d\sigma(s)}{z-s},
\]
where \( \sigma \) is a compactly supported Borel measure on the real line. Since
\[
-\widehat\sigma(z)= \int{s}d\mu(s) + \int\frac{1+zs}{s-z}d\mu(s),
\]
where \( d\mu(s) =(1+s^2)^{-1}d\sigma(s) \), we conclude that $-\widehat\sigma(z)$ is a Nevanlinna-Pick function. Observe that every Nevanlinna-Pick function  satisfying  $\supp\, \mu\, \cap\, [a,b]=\emptyset$ is necessarily real on $[a,b]$ and analytic in some neighborhood of this interval. 

In another direction, a function \( F(z) \) is called a \emph{P\'olya frequency series} if its Taylor coefficients at the origin form a totally positive sequence. More precisely, if
\begin{align}
\label{PFS-Taylor}
F(z) = a_0 + a_1z + \cdots + a_nz^n + \cdots,
\end{align}
then every minor of the Toeplitz matrix \( (a_{i-j})_{i,j=0}^\infty \) is non-negative, where we  set \( a_n=0\) for \( n<0\); see \cite{MR23309}. Schoenberg proved \cite{MR23309} that every function of the form
\begin{align}
\label{PFS}
F(z) = cz^\lambda e^{\gamma z}\prod_{i=1}^\infty \frac{1+\alpha_iz}{1-\beta_i z},
\end{align}
where \( c,\lambda,\gamma,\alpha_i,\beta_i\geq 0 \), \( \lambda \) is a non-negative integer, and \( \sum_i(\alpha_i+\beta_i)<\infty\), is a P\'olya frequency series. The converse  was later established by Edrei \cite{MR53176}; see also \cite{MR53174}. That is, \( F(z) \) is a P\'olya frequency series if and only if it admits representation \eqref{PFS}.

\begin{theorem}
\label{thm:1}
Assume that a function \( F(x) \) is not rational and is
\begin{itemize}
    \item [(i)] either a Nevanlinna-Pick function \eqref{N-P} corresponding to a finite Borel measure \( \mu \) whose support is disjoint from an interval \( [a,b] \);
    \item [(ii)] or a P\'olya frequency series  \eqref{PFS} satisfying \( [a,b]\subset (r_{\mathrm z},r_{\mathrm p}) \), where \( r_{\mathrm z} \) and \( r_{\mathrm p} \) denote, respectively, the largest zero and the smallest pole of \( F(x) \) (\( r_{\mathrm z}=-\infty \) if \( F(x)\) has no zeros and \( r_{\mathrm p}=\infty \) if \( F(x)\) has no poles).
\end{itemize}
Given \( (\nu_n,T_n) \) as in \eqref{nu_n and T_n}, the corresponding AAA approximant \( R_n(x) \) of \( F(x) \) is uniquely determined by \eqref{AAA} and \eqref{L2-min}. Set 
\[
P_n(z) := \sum_{k=1}^n b_k L_k(z) \quad \text{and} \quad Q_n(z) := \sum_{k=1}^n b_k F(t_k) L_k(z)\,,
\]
where \( L_k(z) := \prod_{i\neq k}(z-t_i) \), so that \( R_n(x)=Q_n(x)/P_n(x) \). Then, there exist additional interpolation points \( a<t_1^\prime<\ldots< t_{n-1}^\prime<b \) such that 
\begin{align}\label{cot1}
(P_nF-Q_n)V_n^{-1}\in\mathcal H[a,b], \quad V_n(z) := \prod_{k=1}^n(z-t_k)\prod_{k=1}^{n -1}(z-t_k^\prime).
\end{align}
\end{theorem}

\begin{figure}[ht!]
    \centering
    \includegraphics[width=.9\linewidth]{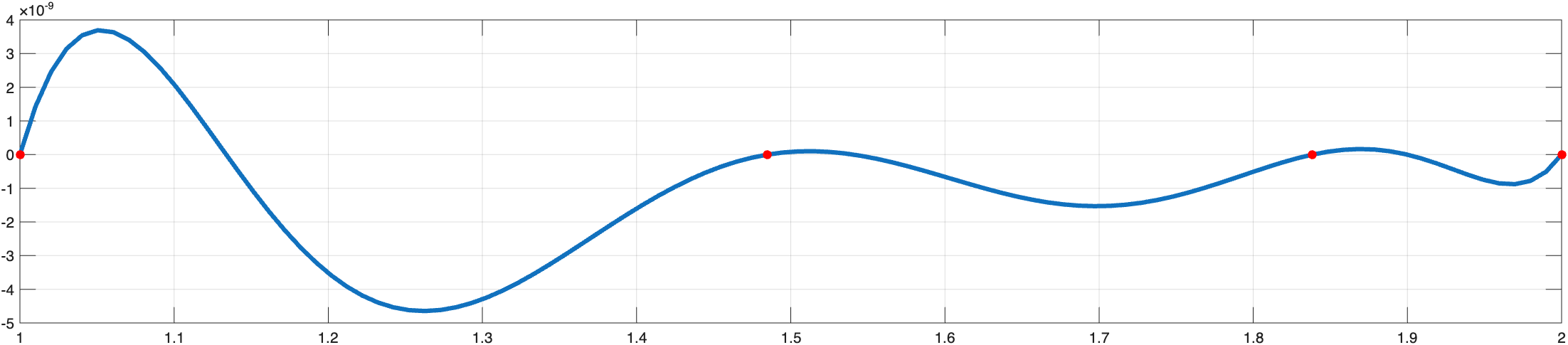} \\
    \includegraphics[width=.9\linewidth]{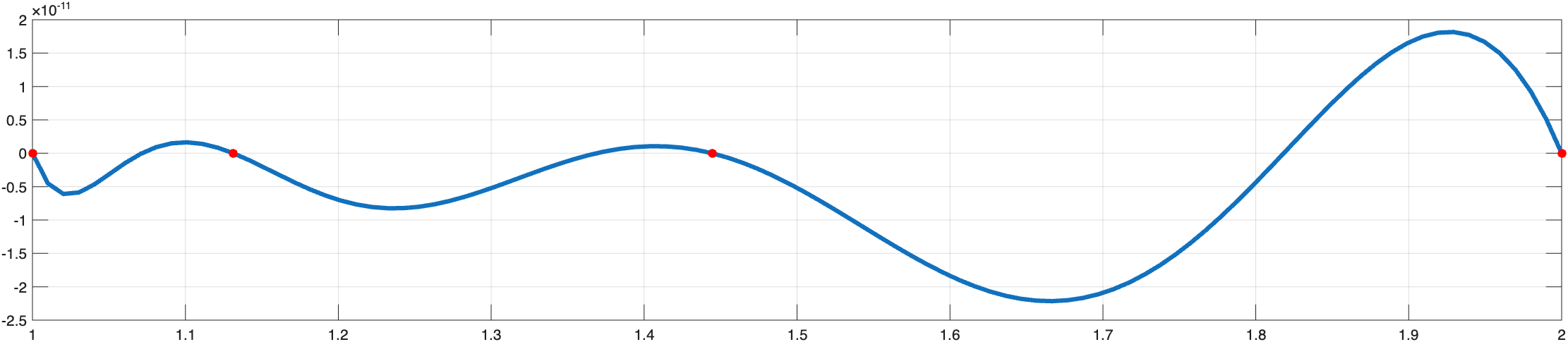}
    \caption{\small The error \( R_4(x)-f(x)\) for \( f(x)=e^x/(10+x)\) (upper graph) and \( f(x) = \ln x- \ln (1+x)\) (lower graph). Red dots are the corresponding interpolation sets \( T_4 \) (chosen through a greedy algorithm as described before). The measure \( \nu_4 \) is a sum of delta masses on 100 equally spaced points on \( [1,2]\).}
    \label{fig:step 1}
\end{figure}

As observed in \cite[Proposition~3.1]{MR3805855}, the AAA algorithm is invariant under linear transformations. In particular, Theorem~\ref{thm:1} applies to affine transformations of the Nevanlinna-Pick functions and P\'olya frequency series as well. On the other hand, multi-point Pad\'e approximants are invariant not only under affine but also under linear fractional transformations. This additional invariance  is particularly useful because every Nevanlinna-Pick function covered  by Theorem~\ref{thm:1} can be transformed into a Markov function by a suitable linear fractional transformation. Since the convergence theory for multi-point Pad\'e approximants to Markov functions is well developed (see \cite[Theorem~1]{GL}, \cite[Lemma~6.1.2]{StahlTotik}, and \cite[Theorem~2]{St00}), one readily obtains the following result.

\begin{corollary}
\label{c2}
Let \( F(z) \) be a non-rational Nevanlinna-Pick function \eqref{N-P} associated with a finite Borel measure \( \mu \) and let \( [a,b] \) be an interval disjoint from \( \supp\,\mu\). Let \( R_n(z) \) be the AAA approximant corresponding to some pair \( (\nu_n,T_n) \) as in \eqref{nu_n and T_n}. Then, \( R_n(z)\) is irreducible, that is, \( P_n(z) \) and \( Q_n(z) \) have no common zeros. Moreover, all the zeros of \( P_n(z)\), equivalently the poles of \( R_n(z)\), belong to \( (-\infty,a_\mathrm{max}] \cup [b_\mathrm{max},\infty) \), where \( (a_\mathrm{max}, b_\mathrm{max}) \) is the largest interval containing \( [a,b] \) that is disjoint from \( \supp\,\mu \). 

Furthermore, let \( \{ R_n(z) \}\) be a sequence of AAA approximants associated with a sequence of pairs \( \{(\nu_n,T_n)\}_{n\ge 1}\) as above. Then, for every compact set \( K \subset \C\setminus \big((-\infty,a_\mathrm{max}] \cup [b_\mathrm{max},\infty) 
\big)\), there exists a constant \( 0<q_K<1 \) such that
\[
|F(z)-R_n(z)| \leq q_K^n
\]
for all sufficiently large \(n\) and all $z\in K$.
\end{corollary}

The proof of Corollary~\ref{c2} essentially consists of showing how the results of \cite{GL} can be adapted to the present setting. If, in addition, we knew that the transformed Nevanlinna-Pick function is a Markov function of a Szeg\H{o} measure, then we could use \cite{St00} to strengthen the last result of Corollary~\ref{c2} by deriving strong asymptotics for the error of approximation.

Similarly, the convergence of the multi-point Pad\'e approximants to P\'olya frequency series with finitely many zeros and poles  is known \cite{MR1370511} when the interpolating points are distributed on a given interval; see also \cite{Frank3,Frank1,Frank2} for further results on the approximation of the exponential function. At present, no such result is available for an arbitrary P\'olya frequency series (however, convergence of the classical Pad\'e approximants, i.e., when all the interpolation points are at the origin, is  known \cite{MR276452}). In view of Theorem~\ref{thm:1}, the following corollary is  an immediate restatement of \cite[Theorems~2.2 and~2.9]{MR1370511}.

\begin{corollary}
Let \( F(z) \) be a P\'olya frequency series of the form
\[
F(z) = e^z \frac{\prod_{i=1}^k(1+\alpha_i z)}{\prod_{i=1}^l(1-\beta_i z)},
\]
where \( k,l \) are natural numbers. Given any sequence of pairs \( \{(\nu_n,T_n)\}_{n\ge 1}\) as in \eqref{nu_n and T_n} for a fixed interval \( [a,b] \), let \( \{ R_n(z) \}\) denote the corresponding sequence of AAA approximants. Then, \( R_n(z)\) is irreducible for all sufficiently large \( n \). Moreover,
\[
\lim_{n\to\infty} R_n(z) = F(z)
\]
locally uniformly in \( \C\setminus\cup_{i=1}^l\{1/\beta_i \} \). Furthermore, if \( F(z)=e^z \), then for every compact set \( K \) there exist constants \( 0<c_K<C_K<\infty \) such that
\[
c_K|V_n(z)| \leq \frac{(2n-2)!(2n-1)!}{(n-1)!(n-1)!}|e^z-R_n(z)| \leq C_K|V_n(z)|
\]
for all sufficiently large \( n \)  and all \( z\in K \), where \( V_n(z) \) is defined by \eqref{cot1}.
\end{corollary}

\section{Sign-Regular Kernels}

Let \( K(x,y)\) be an infinitely differentiable kernel defined on a rectangle \( X\times Y \), where \( X,Y \) are open intervals. We recall the following definitions from \cite[Section~2.1]{Karlin}. The kernel \( K(x,y)\) is said to be \emph{strictly sign-regular} if, for every natural number \( m \), there exists \( \varepsilon_m\in\{-1,1\} \) such that
\begin{align}
\label{ssr}
\varepsilon_m\det\big(  K(x_l,y_k) \big)_{l,k=1}^m >0
\end{align}
whenever \( x_1<x_2<\ldots<x_m\) and \( y_1<y_2<\ldots <y_m\), where  \(x_i\in X\) and \( y_i\in Y \). When \( \varepsilon_m=1 \) for every $m$, the kernel \( K(x,y)\) is called \emph{strictly totally positive}. For each $m$, define the {\it Wronskian determinant} of  \( K(x,y) \) by
\begin{align}
\label{Wm}
W_m(x,y) := \det\left( K^{(i,j)}(x,y) \right)_{i,j=0}^{m-1}, \quad K^{(i,j)}(x,y) := \frac{\partial^{i+j}}{\partial x^i\partial y^j} K(x,y)\,.
\end{align}
 The kernel \( K(x,y)\) is said to be \emph{extended sign-regular} if, for all natural numbers \( m \), there exists  \( \varepsilon_m\in\{-1,1\} \) such that  
\begin{align}
\label{wron}
\varepsilon_mW_m(x,y)>0,
\end{align}
 for all \( x\in X,y\in Y\).  Again, the kernel \( K(x,y)\) is called \emph{extended totally positive} if \( \varepsilon_m=1\) for all \( m \). The extended sign-regularity implies strict sign-regularity, i.e., \eqref{wron} implies \eqref{ssr}, see \cite[Theorem~~2.6]{Karlin}.

Sign-regular kernels naturally give rise to \emph{T-systems} (Chebysh\"ev systems). Indeed, fix a natural number \( n \), points \( y_1<y_2<\ldots <y_n \), and define
\[
\varphi_k(x) := K(x,y_k), \quad k\in\{1,\ldots,n\}.
\]
Then, \eqref{ssr} with \( m=n \) says that \( \varphi_1(x),\ldots,\varphi_n(x)\) is a T-system by definition, see \cite[Definition~III.1.3]{Pinkus}. In fact, since we require \eqref{ssr} to hold for all \( m \), it also follows that
\begin{align}
\label{descartes}
\varepsilon_m\det\big(\varphi_{i_l}(x_k)\big)_{l,k=1}^m>0
\end{align}
for every \( 1\leq m\leq n \), every increasing sequence of indices \( 1\leq i_1<i_2<\ldots<i_m\leq n \), and every choice of points \( x_1<x_2<\ldots<x_m \). That is, every sub-collection of \( \varphi_1(x),\ldots,\varphi_n(x)\) is itself a T-system. A linear combination of the form
\begin{align}
\label{fg1}
\varphi(x) = c_1\varphi_1(x) + \cdots + c_n\varphi_n(x), \quad c_k\in\R,
\end{align}
is called a generalized polynomial. It follows from \eqref{ssr}, see \cite[Proposition~III.1.3]{Pinkus}, that every non-trivial generalized polynomial \( \varphi(x)\) can have at most \( n-1 \) zeros on \( X \). If, in addition,  \eqref{wron} holds, then every non-trivial \( \varphi(x)\) has at most \( n-1 \) zeros on \( X \) counted  with multiplicity, see \cite[Proposition~III.1.6]{Pinkus} (recall that \( K(x,y)\) is infinitely differentiable by assumption and therefore so are the generalized polynomials).

The following observations will be useful in the proof of the forthcoming proposition. Take functions \( \varphi_{i_1}(x),\ldots,\varphi_{i_m}(x)\) where $i_1<\ldots<i_m$ and fix points \( x_1<x_2<\ldots<x_{m-1} \). Assume \eqref{ssr} and consider the function
\begin{align}
\label{poly phi}
\phi(x) = \det \begin{pmatrix} \varphi_{i_1}(x_1) & \varphi_{i_2}(x_1) & \cdots & \varphi_{i_m}(x_1) \\ \vdots & \vdots & \ddots & \vdots \\ \varphi_{i_1}(x_{m-1}) & \varphi_{i_2}(x_{m-1}) & \cdots & \varphi_{i_m}(x_{m-1}) \\ \varphi_{i_1}(x) & \varphi_{i_2}(x) & \cdots & \varphi_{i_m}(x) \end{pmatrix}\,. 
\end{align}
Since $\phi(x)$ is a non-trivial  linear combination of \( m \) functions forming a T-system, it
has exactly \( m-1 \) zeros, namely \( x_1,\ldots,x_{m-1} \). If we additionally assume \eqref{wron}, then these zeros must necessarily be simple. In particular, \( \phi^\prime(x_k) \neq 0 \) for each \( k\in\{1,\ldots,m-1 \} \). Furthermore, 
\[
\sgn\, \phi^\prime(x_k) = \sgn \, \det \begin{pmatrix} \varphi_{i_1}(x_1) & \varphi_{i_2}(x_1) & \cdots & \varphi_{i_m}(x_1) \\ \vdots & \vdots & \ddots & \vdots \\ \varphi_{i_1}(x_{m-1}) & \varphi_{i_2}(x_{m-1}) & \cdots & \varphi_{i_m}(x_{m-1}) \\ \varphi_{i_1}^\prime(x_k) & \varphi_{i_2}^\prime(x_k) & \cdots & \varphi_{i_m}^\prime(x_k) \end{pmatrix},
\]
which implies that
\[
\sgn\, \phi^\prime(x_k) = \sgn \, \det \begin{pmatrix} \varphi_{i_1}(x_1) & \varphi_{i_2}(x_1) & \cdots & \varphi_{i_m}(x_1) \\ \vdots & \vdots & \ddots & \vdots \\ \varphi_{i_1}(x_{m-1}) & \varphi_{i_2}(x_{m-1}) & \cdots & \varphi_{i_m}(x_{m-1}) \\ \varphi_{i_1}(x_k+h) & \varphi_{i_2}(x_k+h) & \cdots & \varphi_{i_m}(x_k+h) \end{pmatrix}
\]
for all sufficiently small \( h>0 \). The sign of the last determinant is determined by \eqref{descartes} and is independent of the choice of the increasing indices \( i_1,\ldots,i_m \). 

\begin{proposition}
\label{prop:orthogonality}
Assume \( K(x,y) \) satisfies \eqref{wron}. Define 
\[
\varphi_*(x) := a_1\varphi_1(x) + \cdots + a_n\varphi_n(x),
\]
where  \( a_1,\ldots,a_n\) are fixed nonzero real numbers satisfying
\begin{align}
\label{sign ak}
\sgn\, a_k = (-1)^{k-1}, \quad k\in\{1,\ldots,n\}.
\end{align}
Let \( \nu \) be a Borel measure on \( [a,b]\subset X \) whose support contains at least \( n \) points. 
Assume that 
\begin{align}
\label{ortho}
\int \varphi_*(x)\varphi(x) d\nu(x) =0    
\end{align}
for every generalized polynomial \( \varphi(x) \) whose coefficients \( c_1,\ldots,c_n\) satisfy
\begin{align}
\label{vector ortho}
a_1c_1 + \cdots + a_nc_n =0.
\end{align}
Then, \( \varphi_*(x) \) has exactly \( n-1 \) zeros that all belong to the interval \( (a,b) \) and  are all simple.
\end{proposition}
\begin{proof}
We begin by showing  that a generalized polynomial \( \varphi(x) \) with coefficients satisfying \eqref{vector ortho} can be prescribed to vanish at any  $j$ distinct points of $X$, where \(j\in\{1,\ldots, n-2\} \), without vanishing anywhere else on $X$. To prove this, fix distinct points \( x_1,\ldots,x_{m-2} \) on \( X \), \(  2\leq m\leq n \) (when \(m=2\), we fix no points). Since the linear system 
\begin{align}
\label{choosing phi}
\begin{cases}
a_1c_1 + \cdots + a_mc_m =0, \\ 
c_{m+1} = \cdots = c_n =0, \\
c_1\varphi_1(x_l) + \cdots + c_m\varphi_m(x_l) =0, \quad 1\leq l\leq m-2
\end{cases}
\end{align}
consists of $n-1$ homogeneous equations with $n$ unknowns $c_1,\ldots,c_n$, it admits a non-trivial solutions (when \( m=2\), the bottom row of \eqref{choosing phi} is vacuous). Let \( \varphi(x) \) be the generalized polynomial corresponding to a solution of \eqref{choosing phi}. Assume to the contrary that there exists \( x_{m-1} \in X\setminus\{x_1,\ldots,x_{m-2} \} \) such that \(\varphi(x_{m-1})=0 \). In this case
\begin{align}
\label{system}
\begin{pmatrix}
a_1 & a_2 & \cdots & a_m \\
\varphi_1(x_1) & \varphi_2(x_1) & \cdots & \varphi_m(x_1) \\
\vdots & \vdots & \ddots & \vdots \\
\varphi_1(x_{m-1}) & \varphi_2(x_{m-1}) & \cdots & \varphi_m(x_{m-1})
\end{pmatrix} 
\begin{pmatrix}
c_1 \\ c_2 \\ \vdots \\ c_m    
\end{pmatrix}
= \begin{pmatrix}
0 \\ 0 \\ \vdots \\ 0    
\end{pmatrix},
\end{align}
where, for convenience, we relabel the zeros so that \( x_1<\cdots<x_{m-1} \).  Expanding the determinant of the coefficient matrix along its first row and  multiplying by \( \varepsilon_{m-1} \) gives us
\begin{align}
\label{contradiction}
\varepsilon_{m-1}\sum_{i=1}^m |a_i| \, \det \big(\varphi_l(x_k)\big)_{l\neq i,k\in \{1,\ldots,m-1\}}>0,    
\end{align}
where we used \eqref{descartes} and \eqref{sign ak} for the last conclusion. Hence, the determinant is non-zero and \eqref{system} can have only the trivial solution contradicting the choice of the coefficients $c_1,\ldots,c_m$.  Thus, any non-trivial \( \varphi(x) \) satisfying \eqref{choosing phi} cannot have any other zeros on $X$ besides $x_1,\ldots,x_{m-2}$ (no zeros at all if \( m=2\)).

 Next, when \( m>2 \), we claim that all of these $m-2$ zeros are simple. Indeed, assume to the contrary that \( x_k \) is not a simple zero, i.e., \( \varphi^\prime(x_k) =0, 1\le k\le m-2 \). Then,
\begin{align}
\begin{pmatrix}
\label{system2}
a_1 & a_2 & \cdots & a_m \\
\varphi_1(x_1) & \varphi_2(x_1) & \cdots & \varphi_m(x_1) \\
\vdots & \vdots & \ddots & \vdots \\
\varphi_1(x_{m-2}) & \varphi_2(x_{m-2}) & \cdots & \varphi_m(x_{m-2}) \\
\varphi_1^\prime(x_k) & \varphi_2^\prime(x_k) & \cdots & \varphi_m^\prime(x_k)
\end{pmatrix} 
\begin{pmatrix}
c_1 \\ c_2 \\ \vdots \\ c_{m-1} \\ c_m    
\end{pmatrix}
= \begin{pmatrix}
0 \\ 0 \\ \vdots \\ 0 \\ 0     
\end{pmatrix}.    
\end{align}
Define functions \( \phi_l(x), l\in \{1,\ldots,m\} \), by 
\begin{align}
\label{poly phi-l}
\phi_l(x) = \det \begin{pmatrix} \varphi_{1}(x_1)&\cdots & \varphi_{l-1}(x_1) &\varphi_{l+1}(x_1)& \cdots & \varphi_{m}(x_1) \\ \vdots & \ddots&\vdots & \vdots &\ddots& \vdots \\ \varphi_{1}(x_{m-2}) &\cdot& \varphi_{l-1}(x_{m-2})&\varphi_{l+1}(x_{m-2}) & \cdots & \varphi_{m}(x_{m-2}) \\ \varphi_{1}(x) & \cdots&\varphi_{l-1}(x) & \varphi_{l+1}(x)&\cdots & \varphi_{m}(x) \end{pmatrix}\,,
\end{align}
(which is  \eqref{poly phi} with \( m \) replaced by \( m-1 \) and the corresponding indices \( \{i_1,\ldots,i_{m-1}\} \) given by \( \{1,\ldots,m\}\setminus\{l\} \)).  Then, expanding the determinant of the matrix in \eqref{system2} in the first row gives 
\begin{equation}\label{ff}
 \sum_{i=1}^m |a_i| \phi_i^\prime(x_k)=
\upsilon_k \sum_{i=1}^m |a_i| |\phi_i^\prime(x_k)|,
\end{equation}
where \( \upsilon_k=\sgn\,\phi_i^\prime(x_k) \) is independent of \( i \) by the observation made right before this proposition. The quantity in the right-hand side of \eqref{ff} is non-zero so, again, \eqref{system2} can have only the trivial solution contradicting the choice of the coefficients. Therefore, every \( \varphi(x) \) satisfying \eqref{choosing phi} can have only simple zeros as claimed. 

We now can conclude the proof of the proposition.  Since \( \varphi_*(x) \) is non-trivial, it cannot have more than $n-1$ zeros counted with multiplicity as explained after \eqref{fg1}. Suppose it has fewer that $n-1$ zeros of odd multiplicity on $(a,b)$. By the claims proved above, there is a generalized polynomial $\varphi(x)$ whose coefficients satisfy \eqref{vector ortho} and whose roots are  simple and precisely match the zeros of odd multiplicity of $\varphi_*(x)$. If $\varphi_*(x)$ has no roots of odd multiplicity, that $\varphi(x)$ is chosen to have no roots on $X$. Hence, the product $(\varphi_\ast\varphi)(x)$ keeps the same sign away from the roots of $\varphi_*(x)$. This leads to the contradiction with \eqref{ortho} since  the support of \( \nu \) has at least \( n \) points. 
\end{proof}

Another important observation is that Gram matrices of T-systems coming from strictly sign-regular kernels are \emph{strictly totally positive}; that is, all of their minors are positive, see \cite[Definition~III.2.3]{Pinkus}. Indeed, as in Proposition~\ref{prop:orthogonality}, let \( \nu \) be a Borel measure on \( [a,b] \subset X \) whose support contains at least \( n \) points. Define
\begin{align}
\label{Gram}
G_n(\nu) := \left(\int\varphi_k(x)\varphi_j(x)d\nu(x)\right)_{k,j=1}^n
\end{align}
to be the Gram matrix of the functions \( \varphi_1(x),\ldots,\varphi_n(x)\). Since \( \varepsilon_n^2=1 \) and the support of \( \nu \) has at least \( n \) points, Andr\'eief's identity and \eqref{descartes} readily yield that
\[
\det \, G_n(\nu) = \frac1{n!} \int \cdots \int \det\left(\varphi_i(x_k)\right)_{i,k=1}^n \det\left(\varphi_j(x_k)\right)_{j,k=1}^n \prod_{k=1}^n d\nu(x_k)>0.
\]
Moreover, Andr\'eief's identity and \eqref{descartes} also show that the same conclusion is true for every minor of \( G_n(\nu) \). Therefore, \( G_n(\nu) \) is indeed a strictly totally positive matrix.

It is known, see \cite[Theorem~III.2.8]{Pinkus}, that any \( n\times n\) strictly totally positive matrix has $n$ distinct eigenvalues that are all positive.  Moreover, \( (b_1,\ldots,b_n)^\mathsf{T} \), an eigenvector  corresponding to the smallest eigenvalue,  satisfies
\begin{align}
\label{sign bk}
b_k\neq 0, \;\; k\in\{1,\ldots,n\}, \quad \text{and} \quad \sgn\, b_l = (-1)^{l-1} \sgn \, b_1, \;\; l\in\{2,\ldots,n\}.    
\end{align}
Then, the following corollary is immediate.

\begin{corollary}
\label{cor:ortho} Let  the kernel \( K(x,y) \) satisfy \eqref{wron}
 and  \( (b_1,\ldots,b_n)^\mathsf{T} \) be an eigenvector corresponding to the smallest eigenvalue of \( G_n(\nu) \). Set
$
\varphi(\nu;x) = b_1\varphi_1(x) + \cdots + b_n\varphi_n(x)
$
and suppose that
\[
\int \varphi(\nu;x)\varphi(x) d\nu(x) =0
\]
for every generalized polynomial $\varphi(x)$ whose coefficients $c_1,\ldots,c_n$ satisfy $b_1c_1 + \cdots + b_nc_n =0$.
Then, \( \varphi(\nu;x) \) has exactly \( n-1 \)  zeros that are all simple and contained in \( (a,b) \).
\end{corollary}

\section{Loewner Kernels}
\label{sec:loewner}

Our next goal is to provide examples of those functions \( f(x) \) whose associated Loewner kernel \eqref{Loewner kernel} is extended sign-regular on \( X\times X \) for some open interval \( X \).

We start by assuming that $K_f(x,y)$ is an extended totally positive kernel. Since \eqref{wron} implies \eqref{ssr}, we get that
\begin{align}
\label{loewner}
\det\big(  K_f(s_l,s_k) \big)_{l,k=1}^M >0    
\end{align}
for any \( s_1<s_2<\ldots<s_M \) and any natural number \( M \). By Loewner's theorem (see \cite[Theorems~1.6 and 5.1]{Simon}), \eqref{loewner} with \( > \) replaced by \( \geq  \) is equivalent    to \( f(x) \) being the restriction to \( X \) of a Nevanlinna-Pick function:
\begin{align}
\label{loewner function}
A+ Bz + \int\frac{1+zs}{s-z}d\mu(s),
\end{align}
where \( A \) is real, \( B \) is nonnegative, and \( \mu \) is a finite Borel measure supported on the real line whose support is disjoint from \( X \). Such functions \( f(x) \) are also known as matrix monotone or complete Bernstein functions. In short, extended total positivity of $K_f$ implies \eqref{loewner function} and \eqref{loewner function} implies \eqref{loewner} with \( > \) replaced by \( \geq  \). The converse statement that non-rational Nevanlinna-Pick functions give rise to extended totally-positive kernels essentially follows from the material in \cite[Section~3.1]{Karlin}. For completeness, we provide a short proof below.

\begin{proposition}
\label{prop:etp}
Let \( F(z) \) be a non-rational Nevanlinna-Pick function having the integral representation \eqref{loewner function}. Assume that the support of \( \mu \) is disjoint from some open interval \( X \). Then, the Loewner kernel \( K_F(x,y) \) is extended totally positive on \( X\times X \).
\end{proposition}
\begin{proof}
Let \( d\sigma(s) = (1+s^2)d\mu(s) \). Then,
\[
K_F(x,y): = B + \int\frac{d\sigma(s)}{(s-x)(s-y)}, \quad x,y\in X.
\] 
Clearly, $K_F(x,y)>0$, which proves \eqref{wron} for \( m=1\) and $\varepsilon_1=1$.  Assume first that \( B= 0 \).  Differentiating under the integral sign gives
\[
K_F^{(i,j)}(x,y) =  i!j!\int\frac{d\sigma(s)}{(s-x)^{i+1}(s-y)^{j+1}}.
\]
Since the factors \( i! \) and \( j! \) can be understood as coming from both left and right multiplication by the diagonal matrix of factorials, we get from Andr\'eief's identity that
\[
W_m(x,y) = \frac{G^2(m+1)}{m!} \int\cdots\int \det\left( (s_k-x)^{-i} \right)_{i,k=1}^m \det\left( (s_k-y)^{-j} \right)_{j,k=1}^m  \prod_{k=1}^m d\sigma(s_k),
\]
where \( G\) denotes the Barnes \( G \)-function. The formula for the Vandermonde determinant now gives that
\begin{align}
\det\left( \frac1{(s_k-x)^i} \right)_{i,k=1}^m & = \frac{\det\left( (s_k-x)^{m-i} \right)_{i,k=1}^m}{\prod_{k=1}^m (s_k-x)^m} = \frac{\prod_{1\leq i<j\leq m}(s_i-s_j)}{\prod_{k=1}^m (s_k-x)^m}
\end{align}
(notice the opposite order of the rows as compared to the standard Vandermonde determinant). Thus,
\[
W_m(x,y) = \frac{G^2(m+1)}{m!} \int\cdots\int\frac{\prod_{1\leq i<j\leq m}(s_i-s_j)^2}{\prod_{k=1}^m (s_k-x)^m(s_k-y)^m}\prod_{k=1}^m d\sigma(s_k),
\]
which is clearly positive as \( \sigma \) has support of infinite cardinality. Hence,  $K_F$ is extended totally positive  when $B=0$. Consider the case when $B>0$. In the formula for $W_m(x,y)$, $B$ appears only in the \( (1,1) \)-entry.  Observe that our arguments above remain valid if we replace \( K_F(x,y) \) with
\[
K_F^{(1,1)}(x,y) = \int\frac{d\sigma(s)}{(s-x)^2(s-y)^2},
\]
which shows that the principal minor corresponding to the indices \( \{2,\ldots,m\} \) of the determinant defining $W_m(x,y)$  is positive. Hence, expanding this determinant in the first row, we get  $W_m(x,y)>0$ even when \( B>0 \) so that $K_F$ is extended totally positive.
\end{proof}

Thus, Nevanlinna-Pick functions are exactly those whose Loewner kernels are extended totally positive. The technique we used to establish this claim extends to the exponential function as well. Indeed, if \( e(x)=e^x \), then
\[
K_e(x,y) = \int_0^1 e^{sx+(1-s)y} ds \quad \Rightarrow \quad K_e^{(i,j)}(x,y) = \int_0^1 s^i(1-s)^j e^{sx+(1-s)y} ds.
\]
Andr\'eief's identity and the Vandermonde determinant formula then yield
\begin{align}
W_m(x,y) & = \frac1{m!} \int_0^1\cdots\int_0^1 \det\left( s_k^{i-1} \right)_{i,k=1}^m \det\left( (1-s_k)^{j-1} \right)_{j,k=1}^m  \prod_{k=1}^m e^{s_kx+(1-s_k)y} ds_k \\
& = \frac{(-1)^{\frac{m(m-1)}2}}{m!} \int_0^1\cdots\int_0^1 \prod_{1\leq i<j\leq m}(s_j-s_i)^2  \prod_{k=1}^m e^{s_kx+(1-s_k)y} ds_k,
\end{align}
where the integrand is clearly a non-negative function.

As the above formula shows,  the sign pattern of the Wronskians for the exponential function is \( +--++--++--\cdots\).  A natural question is for which other functions the Wronskians share this sign pattern. To address it, we start with the observation that 
\[
K_f^{(i,j)}(x,y) = \int_0^1 s^i(1-s)^jf^{(i+j+1)}(sx+(1-s)y)ds.
\]
This identity and the formula for the beta function evaluated at natural numbers give
\begin{align}
\label{der K}
K_f^{(i,j)}(x,x) = \frac{i!j!}{(i+j+1)!} f^{(i+j+1)}(x) = i!j! a_{i+j+1}(x),
\end{align}
where \( a_n(x) \) are the Taylor coefficients at the origin of \( f(z+x)\) considered as a function of \( z \). Hence, the Wronskian \( W_m(x,y)\) along the diagonal \( x=y\) can be expressed as
\begin{align}
W_m(x,x) &= G^2(m+1) \det\big( a_{i+j+1}(x) \big)_{i,j=0}^{m-1} \\
& = (-1)^{\frac{m(m-1)}2}G^2(m+1)\det\big( a_{i-j+m}(x) \big)_{i,j=1}^m,
\end{align}
where the last equality is obtained by reversing the order of the columns in the determinant. Thus, if we want to have
\begin{align}
\label{sign pattern}
(-1)^{\frac{m(m-1)}2}W_m(x,x)>0,
\end{align}
which is the desired sign pattern, all Toeplitz determinants \( \det( a_{i-j+m}(x))_{i,j=1}^m \) must be positive for all \( m\geq 1 \).  In \cite{MR1545467}, Schoenberg proved that if 
\begin{align}
\label{property P}
\det\big( a_{i-j+m}(x) \big)_{i,j=1}^n>0, \quad m,n-1 \geq 0,
\end{align}
then \( \{ a_m(x) \}_{m\geq0} \) is a totally positive sequence and therefore, as explained in the introduction, the function \( f(z+x)\) must be a P\'olya frequency series \eqref{PFS}. In fact, as pointed out in footnote 2 on page 87 of \cite{MR53176}, if \( f(z+x) \) is a non-rational P\'olya frequency series, then \eqref{property P} holds. Notice also that if \( F(z) \) is a P\'olya frequency series and \( x\in X \), where \( X \) is an open interval between the largest zero and the smallest pole  of \( F(z) \), then \( F(z+x) \) is also a P\'olya frequency series. In this case, assuming \( F(z) \) to be non-rational, \eqref{property P} and respectively \eqref{sign pattern} hold for all \( x\in X \). In the following proposition we prove that \eqref{sign pattern} remains valid if \( W_m(x,x) \) is replaced by \( W_m(x,y) \) for \( x,y\in X \).

\begin{proposition}
\label{prop:esr}
Let \( F(z) \) be a non-rational P\'olya frequency series \eqref{PFS}. Further, let \( X \) be an open interval between the largest zero of \( F(z) \) and its smallest pole (\( X \) is unbounded on the left if \( F(z) \) has no zeros and is unbounded on the right if \( F(z) \) has no poles). Then, the Loewner kernel \( K_F(x,y) \) is extended sign-regular on \( X\times X \) with the sign pattern \( +--++--++--\cdots\). 
\end{proposition}
\begin{proof}
Let \( r_{\mathrm z} \) be the largest zero of \( F(z) \) (necessarily \( r_{\mathrm z}\leq0 \)) or an arbitrary negative number if \( F(z) \) has no zeros. As mentioned before the proposition, \( F(z+r_{\mathrm z})\) is also a P\'olya frequency series. Hence, it is enough to prove the desired claim for \( X=(0,r_{\mathrm p})\), where \( r_{\mathrm p} \) is the smallest pole of \( F(z) \) and necessarily is also the radius of convergence of the series~\eqref{PFS-Taylor} (\( r_{\mathrm p}=\infty \) if \( F(z)\) has no poles). Fix a natural number \( l\geq 1 \). By definition, 
\begin{align}
\label{W det}
W_l(x,y) = \det\left( K^{(i-1,j-1)}(x,y)\right)_{i,j=1}^l = \sum_\sigma (-1)^\sigma \prod_{i=1}^l K^{(i-1,\sigma(i)-1)}(x,y),
\end{align}
where the sum is taken over all the permutations \( \sigma \) of \( \{1,2,\ldots,l\} \) and \( (-1)^\sigma \) denotes the sign of the permutation. It follows from the Taylor theorem for analytic functions in two variables that 
\begin{align}
\label{W Taylor}
W_l(x,y) = \sum_{n,m=0}^\infty w_{l;n,m}x^ny^m, \quad w_{l;n,m} = \frac{W_l^{(n,m)}(0,0)}{n!m!},
\end{align}
where the series is convergent for all \( |x|,|y|<r_{\mathrm p} \). To prove the proposition, it is sufficient to show that the above Taylor coefficients \( w_{l;n,m} \) all have the required sign pattern: $\sgn \,w_{l;n,m}=(-1)^{l(l-1)/2}$. We apply the generalized Leibniz rule to \eqref{W det} to conclude  that \( W_l^{(n,m)}(x,y) \) is equal to
\begin{align}
\sum_{|\boldsymbol q|=m}\sum_{|\boldsymbol p|=n}\frac{m!}{q_1!\cdots q_l!}\frac{n!}{p_1!\cdots p_l!}\sum_\sigma (-1)^\sigma \prod_{i=1}^l K^{(p_i+i-1, q_{\sigma(i)} + \sigma(i) -1)}(x,y),
\end{align}
where \( \boldsymbol p=(p_1,\ldots,p_l) \), \( |\boldsymbol p|=p_1 + \cdots + p_l\), and \( p_i\geq 0 \) (the same notation holds for \( \boldsymbol q\)). Let, as before, \( a_n \) be the Taylor coefficients of \( F(z) \) at the origin. By evaluating the above expression at \( (0,0) \), dividing it by \( n!m! \), and using \eqref{der K}, we obtain
\begin{align}
w_{l;n,m} &= \sum_{|\boldsymbol q|=m}\sum_{|\boldsymbol p|=n} \left(\prod_{s=1}^l \frac{(p_s+s-1)!}{p_s!}\frac{(q_s+s-1)!}{q_s!}\right) \det\big(a_{p_i+i+q_j+j-1}\big)_{i,j=1}^l \\
\label{wnm}
& = \sum_{|\boldsymbol q|=m}\sum_{|\boldsymbol p|=n} \left(\prod_{s=1}^l \frac{\Gamma(p_s+s)}{\Gamma(p_s+1)}\frac{\Gamma(q_s+s)}{\Gamma(q_s+1)}\right) \det\big(a_{p_i+i+q_j+j-1}\big)_{i,j=1}^l,
\end{align}
where convenience of the second representation will become apparent momentarily. To proceed,  notice that, up to a sign, several of the summands above correspond to the same Hankel determinant; that is, their corresponding matrices differ only by row and/or column permutations. Hence, our next goal is to bring all determinants that are identical up to sign into a standard (Hankel) form. Suppose first that \( p_i+i=p_k+k \) for some \( i \) and \( k, i\neq k \), then the corresponding determinant vanishes since the underlying matrix has two identical rows. Otherwise, let \( \boldsymbol p \) be such that all \( p_i+i \) are all distinct. Then, there exists a permutation \( \sigma_{\boldsymbol p} \) such that
\[
p_{\sigma_{\boldsymbol p}(1)}+\sigma_{\boldsymbol p}(1)<p_{\sigma_{\boldsymbol p}(2)}+\sigma_{\boldsymbol p}(2)< \cdots < p_{\sigma_{\boldsymbol p}(l)}+\sigma_{\boldsymbol p}(l).
\]
Define \( u_i = p_{\sigma_{\boldsymbol p}(i)}+\sigma_{\boldsymbol p}(i)-i \). It holds that \( u_1 = p_{\sigma_{\boldsymbol p}(1)}+\sigma_{\boldsymbol p}(1)-1\geq p_{\sigma_{\boldsymbol p}(1)} \geq 0 \) and
\[
u_{i+1} = p_{\sigma_{\boldsymbol p}(i+1)}+\sigma_{\boldsymbol p}(i+1)-i - 1 > p_{\sigma_{\boldsymbol p}(i)}+\sigma_{\boldsymbol p}(i)-i -1 = u_i - 1.
\]
Hence, \( u_{i+1}\geq u_i\geq \ldots \ge u_1\geq 0 \). Observe also that
$
|\boldsymbol u|=|\boldsymbol p|=n\,.
$
Denote by $\boldsymbol \Upsilon$ the collection of all vectors $\boldsymbol u=(u_1,\ldots,u_l)$ with integer coordinates that satisfy the properties that $|\boldsymbol u|=n$ and {$0\le u_1\le \ldots \le u_{l-1}\le u_{l}$.} Our argument above shows that every $\boldsymbol p$ corresponds to some element in $\boldsymbol \Upsilon$. If $\boldsymbol u$ is fixed and $\boldsymbol p$ gives rise to such $\boldsymbol u$, we will write 
\( \boldsymbol p\sim \boldsymbol u \). Conversely, suppose \( \boldsymbol u \in \boldsymbol \Upsilon\). Given a permutation \( \sigma \) of \( \{1,\ldots,l\} \), set
\[
p_{\sigma(i)}^\sigma := u_i+i-\sigma(i), \quad i\in\{1,\ldots,l\}.
\]
It is still true that \( p_1^\sigma+\cdots +p_l^\sigma = n \) and that \( p_i^\sigma+i>0 \), but, in general, it is not true that \( p_i^\sigma\geq0 \). However, if \( p_i^\sigma<0 \), then \( \Gamma^{-1}(p_i^\sigma+1)=0 \). Denote \( D(\boldsymbol p,\boldsymbol q) := \det\big(a_{p_i+i+q_j+j-1}\big)_{i,j=1}^l \). Therefore, for fixed $\boldsymbol u$ and $\boldsymbol q$, we get 
\begin{align}
\sum_{\boldsymbol p\sim \boldsymbol u} \left(\prod_{s=1}^l\frac{\Gamma(p_s+s)}{\Gamma(p_s+1)}\right)D(\boldsymbol p,\boldsymbol q)& = \left(\sum_\sigma (-1)^\sigma \prod_{s=1}^l\frac{\Gamma(p_s^\sigma+s)}{\Gamma(p_s^\sigma+1)}\right) D(\boldsymbol u,\boldsymbol q) \\
& = \left(\sum_\sigma (-1)^\sigma \prod_{i=1}^l \frac{\Gamma(u_i+i)}{\Gamma(u_i+i-\sigma(i)+1)}\right) D(\boldsymbol u,\boldsymbol q).
\end{align}
Observe that
\begin{eqnarray*}
W(\boldsymbol u) := 
\sum_\sigma (-1)^\sigma \prod_{i=1}^l \frac{\Gamma(u_i+i)}{\Gamma(u_i+i-\sigma(i)+1)} =
\det\left( \frac{\Gamma(u_i+i)}{\Gamma(u_i+i-k+1)} \right)_{i,k=1}^l.
\end{eqnarray*}
Writing this explicitly in the matrix form gives
\[
W(\boldsymbol u) = \det\begin{pmatrix}
1 & u_1 & \cdots & u_1(u_1-1)\cdots(u_1-l+2) \\
1 & u_2+1 & \cdots & (u_2+1)(u_2)\cdots(u_2-l+3) \\
\vdots & \vdots & \ddots & \vdots \\
1 & u_l+l-1 & \cdots & (u_l+l-1)(u_l+l)\cdots(u_l+1)
\end{pmatrix}.
\]
Upon close examination it becomes clear that the above determinant can be rewritten as the Vandermonde determinant in \( u_1,u_2+1,\ldots,u_l+l-1 \). Hence, since the entries in $\boldsymbol u$ are increasing, we have that
\begin{align}
\label{sign Wu}
W(\boldsymbol u) = \prod_{1\leq i<j\leq l} (u_j-u_i+j-i)>0.
\end{align}
Exactly the same considerations apply to the sum over \( \boldsymbol q \) in \eqref{wnm}. Thus, we get that
\[
w_{l;n,m} = \sum_{\boldsymbol v\in \boldsymbol\Upsilon}\sum_{\boldsymbol u\in \boldsymbol\Upsilon} W(\boldsymbol v)W(\boldsymbol u) D(\boldsymbol u,\boldsymbol v).
\]
Since \( F(z) \) is a non-rational P\'olya frequency series, it then holds that
\begin{align}
\label{sign Duv}
(-1)^{\frac{l(l-1)}{2}}D(\boldsymbol u,\boldsymbol v) = \det\big( a_{u_i+i+v_{l-j+1}+l-j} \big)_{i,j=1}^l>0,
\end{align}
where the last conclusion follows from the fact the determinant in question is a minor of the Toeplitz matrix \( (a_{i-j})_{i,j=1}^\infty \) corresponding to rows \( u_1+1+v_l+l,\ldots,u_l+l+v_l+l\) and columns \( 1,v_l-v_{l-1}+2,\ldots,v_l-v_1+l\). Finally, \eqref{sign Duv} together with \eqref{W Taylor} and \eqref{sign Wu} now shows that 
\[
(-1)^{\frac{l(l-1)}{2}} W_l(x,y)>0,  \quad x,y\in(0,r_{\mathrm r}). \qedhere
\]
\end{proof}

\section{Proof of Theorem~\ref{thm:1} and Corollary~\ref{c2}}

\begin{proof}[Proof of Theorem~\ref{thm:1}]
Let \( F(x) \) be either a non-rational Nevanlinna-Pick function or a non-rational P\'olya frequency series as in the statement of the theorem. Furthermore, let \( R_n(x) \) be a AAA approximant of \( F(x) \) corresponding to the pair \( (\nu_n,T_n) \) as in \eqref{nu_n and T_n}. Given an interval \( [a,b] \) as in the statement of the theorem, it has been shown in Propositions~\ref{prop:etp} and~\ref{prop:esr} of the previous section that the Loewner kernel \( K_F(x,y) \) is extended sign-regular (extended totally positive in the case of Nevanlinna-Pick functions) on some open interval \( X\supset[a,b] \). Define
\[
\varphi_i(x) = K_F(x;t_i), \quad i\in\{1,\ldots,n\}.
\]
Let \( b_1,\ldots,b_n \) be as in \eqref{L2-min} and set
\[
\varphi({\nu_n};x) =  b_1 \varphi_1(x) + \cdots + b_n\varphi_n(x). 
\]
Denote by \( \lambda \) the value of \eqref{L2-min}. The Lagrange multiplier method yields
\begin{equation}
\label{lagrangian}
  \int \varphi(\nu_n;x)\varphi_k(x)d\nu_n(x)=\lambda b_k, \quad \lambda=\int\varphi^2(\nu_n;x)d\nu_n(x).
\end{equation}
From this we infer that $\lambda$ is the smallest positive eigenvalue of the Gram matrix $G_n(\nu_n)$, see \eqref{Gram}, and the vector \( (b_1,\ldots,b_n)^\mathsf{T} \) is a corresponding eigenvector, properly normalized. Since \( K_F(x,y) \) is extended sign-regular, \( G_n(\nu_n) \) is a strictly totally positive as explained right after \eqref{Gram}. Thus,
\begin{align}
\label{sign bk}
\sgn\, b_l = (-1)^{l-1} \sgn \, b_1, \quad l\in\{2,\ldots,n\},    
\end{align}
as noted prior to Corollary~\ref{cor:ortho}. Our normalization $\sgn \, b_1>0$ makes this \( (b_1,\ldots,b_n)^\mathsf{T} \) unique. It further follows from \eqref{lagrangian} that
\begin{align}
\label{orthogonality}    
\int \varphi(\nu_n;x)\varphi(x) d{\nu_n}(x) =0 \quad \text{when} \quad b_1c_1 + \cdots + b_nc_n =0,
\end{align}
where \( \varphi(x) \) is a generalized polynomial corresponding to \( c_1,\ldots,c_n \), see \eqref{fg1}. Then, according to Corollary~\ref{cor:ortho}, \( \varphi(\nu_n;x) \) has exactly \( n-1 \) zeros, all on \( (a,b) \) and all simple. If we denote the zeros of $\varphi(\nu_n;x)$ by \( t_1^\prime,\ldots,t_{n-1}^\prime \), then we get from \eqref{lin error} that
\[
\frac{(P_nF-Q_n)(z)}{V_n(z)} = \frac{\varphi(\nu_n;z)}{\prod_{k=1}^{n-1}(z-t_k^\prime)},
\]
where the right-hand side is clearly analytic around \( [a,b] \), which finishes the proof of the theorem. 
\end{proof}

\begin{proof}[Proof of Corollary~\ref{c2}]

Let $F(z)$, \( R_n(z) \), \( [a,b] \), and $(a_{\max},b_{\max})$ be as stated in the corollary, and let \( P_n(z),Q_n(z) \), and \( V_n(z) \) be as in Theorem~\ref{thm:1}. Take an arbitrary $c\in (a_{\max},a)$ and define the linear fractional transformation \( t(z) := c+z^{-1} \). Then,
\[
F_c(z) := F(t(z)) = A_c + \widehat\sigma_c(z), \quad \supp\, \sigma_c \subseteq I_c:=\big[-(c-a_{\max})^{-1},(b_{\max}-c)^{-1}\big],
\]
for some real constant \( A_c \) and a positive Borel measure $\sigma_c$. Set
\[
P_n^c(z) : = z^{n-1}P_n(t(z)), \;\; Q_n^c(z) := z^{n-1}Q_n(t(z)), \;\;  V_n^c(z) := z^{2n-1} V_n(t(z)).
\]
Then, all of \( 2n-1 \) zeros of \( V_n^c(z) \) (counting multiplicities as \( V_n(z)\) may have double zeros)  belong to the interval \( [a_c,b_c]:=[(b-c)^{-1},(a-c)^{-1}] \). Consequently, \( Q_n^c(z)/P_n^c(z) \) is a multi-point Pad\'e approximant of type \( (n-1,n-1) \) corresponding to the interpolation set formed by the zeros of \( V_n^c(z) \). More precisely,
\[
\frac{(P_n^c F_c-Q_n^c)(z)}{V_n^c(z)} = \frac1{z^n}\frac{(P_n F-Q_n^c)(t(z))}{V_n(t(z))}
\]
is analytic around \( [a_c,b_c] \) (notice that \( a_c>0 \)). Denote the number of potential common zeros of \( P_n^c(z) \) and \( Q_n^c(z) \) by \( n-\tilde n \). Let \( \tilde P_n(z) \) and \( \tilde Q_n(z) \) be polynomials obtained from \( P_n^c(z) \) and \( Q_n^c(z) \) by removing these common zeros. Naturally, these are polynomials of degree at most \( \tilde n-1 \). Similarly, let \( \tilde V_n(z) \) be the polynomial obtained from \( V_n^c(z) \) by removing zeros that are  simultaneous common to \( P_n^c(z) \), \( Q_n^c(z) \), and \( V_n^c(z) \). It necessarily follows that \( n+\tilde n-1\leq \deg \tilde V_n \leq 2n-1\). Then, the expression 
\[
\frac{(\tilde P_nF_c -\tilde Q_n)(z)}{\tilde V_n(z)}
\]
remains analytic around \( [a_c,b_c] \) and, in particular, inside any positively oriented Jordan curve \( \Gamma \) that encircles \( [a_c,b_c] \) and contains \(\supp\, \sigma_c \) in its exterior. Thus, it follows from the Cauchy's theorem that
\[
0 = \int_{\Gamma} s^k\frac{(\tilde P_nF_c -\tilde Q_n)(s)}{\tilde V_n(s)} ds
\]
for any nonnegative integer \( k \). Since \( s^k (A_c\tilde P_n(s)-\tilde Q_n(s))/\tilde V_n(s) \) is analytic in the exterior of \( \Gamma \) and vanishes at infinity with order of at least 2 when \( k+\tilde n-1\leq \deg \tilde V_n-2\), we have
\begin{align}
0 & = -\frac1{2\pi\ic}\int_{\Gamma} s^k\frac{(\tilde P_n\widehat \sigma_c)(s)}{\tilde V_n(s)} ds = \int \left( -\frac1{2\pi\ic}\int_{\Gamma} \frac{s^k \tilde P_n(s)}{\tilde V_n(s)}\frac{ds}{s-x}\right) d\sigma_{c}(x) \\
& = \int x^k \tilde P_n(x) \frac{d\sigma_{c}(x)}{\tilde V_n(x)}, \quad 0\leq k\leq n-2,
\end{align}
by the Cauchy theorem, the Fubini-Tonelli theorem, and the Cauchy integral formula applied in the exterior domain of \( \Gamma \). Since \( \tilde V_n(z) \) maintains a constant sign on \( \supp\, \sigma_c\subset I_c \) (here we use \( I_c \cap [a_c,b_c]=\emptyset \)), \( \tilde P_n(z)\) is a polynomial satisfying \( n-1 \) orthogonality conditions with respect to a positive measure supported on $I_c$. Thus, \( n-1 = \deg \tilde P_n \leq \tilde n-1 \) (in particular, \( \tilde P_n(z)=P_n^c(z) \)) and all the zeros of \( \tilde P_n(x)\) belong to $I_c$. Consequently, \( P_n(z)\) and \( Q_n(z) \) share no common zeros, and all the zeros of \( P_n(z) \) belong to \( (-\infty,a_\mathrm{max}] \cup [b_\mathrm{max},\infty) \).

We now claim that for any closed set \( K \) in \( D_c:=\overline \C\setminus I_c \), there exists a constant \( q_K\in(0,1) \) such that
\begin{align}
\label{convergence}
\big|F_c(z) - Q_n^c(z)/P_n^c(z)\big| \leq q_K^n, \quad z\in K,
\end{align}
for all sufficiently large \( n \). Suppose this claim were false. Then, there would exist sequences \( \{n_i\}\subset\N \), \( \{z_i\}\subset K \), and \( \{\epsilon_i\} \) with \( \epsilon_i\downarrow 0 \) as \( i\to\infty \) such that
\begin{align}
\label{contradiction}
\big|F_c(z_i) - Q_{n_i}^c(z_i)/P_{n_i}^c(z_i)\big| \geq (1-\epsilon_i)^{n_i}.
\end{align}
Denote by \( \nu_n \) the  counting measure of zeros of \( V_n^c(z) \) normalized to be a probability measure. Since all the measures \( \nu_n \) are supported on \( [a_c,b_c] \),  Helly's selection theorem guarantees that \( \{\nu_{n_i}\} \) contains a further subsequence \( \{\nu_{n_{i_k}}\} \)  converging  in  weak$^*$ sense to some probability Borel measure \( \nu \). Then, by \cite[Theorem~6.1.6]{StahlTotik} (see also \cite{GL}), the bound
\begin{equation}\label{sd-h1}
\lim_{k\to\infty}\big|F_c(z) - Q_{n_{i_k}}^c(z)/P_{n_{i_k}}^c(z)\big|^{1/2{n_{i_k}}} \leq e^{-G_{D_c}(\nu;z)}
\end{equation}
holds uniformly on \( K \), where \( G_{D_c}(\nu;z) \) represents the Green potential of the measure \( \nu \) in the domain \( D_c \). Since \( G_{D_c}(\nu;z) \) is positive and superharmonic  in \( D_c \), \eqref{sd-h1} contradicts \eqref{contradiction}. Thus, \eqref{convergence} does take place and yields the last claim of the corollary by replacing \( z \) with \( t^{-1}(z)\).
\end{proof}

\bibliographystyle{plain}

\bibliography{AAA}

\end{document}